\documentclass[12pt, notitlepage,reqno]{amsart}

\usepackage{graphicx}
\usepackage{amsmath,amsthm}
\usepackage{amsfonts}
\usepackage{amssymb}
\usepackage{tikz}
\usepackage{subfig}
\usetikzlibrary{matrix,arrows}
\usepackage{bbm}

\usepackage{arydshln}
\usepackage{mathtools}
\title[Scaffoldings of TP Matrices and Line Insertion]{Scaffoldings of Totally Positive Matrices and Line Insertion}

\newtheorem{thm}{Theorem}[section]
\newtheorem{lem}[thm]{Lemma}

\newtheorem{prop}[thm]{Proposition}
\newtheorem{cor}[thm]{Corollary}

\newtheorem{algorithm}[thm]{Algorithm}
\theoremstyle{definition}
\newtheorem{defn}[thm]{Definition}

\newtheorem{ex}[thm]{Example}

\newtheorem{notn}[thm]{Notation}

\newcommand{\C}{\mathcal}

\newcommand{\vect}{\boldsymbol}

\author{Karel Casteels}
\address{
Department of Mathematics,\newline
University of California, Santa Barbara
}
\email{casteels@ucsb.edu}
\begin{document}

\begin{abstract}

Given a totally positive matrix, can one insert a line (row or column) between two given lines while maintaining total positivity? This question was first posed and solved by Johnson and Smith who gave an algorithm that results in one possible line insertion. In this work we revisit this problem. First we show that every totally positive matrix can be associated to a certain vertex-weighted graph in such a way that the entries of the matrix are equal to sums over certain paths in this graph. We call this graph a scaffolding of the matrix. We then use this to give a complete characterization of all possible line insertions as the strongly positive solutions to a given homogeneous system of linear equations.
	
\end{abstract}

\maketitle
\section{Introduction}

The study of totally positive (TP) and totally nonnegative (TN) matrices, i.e., real matrices where each minor is, respectively, positive and nonnegative, has a long history. See the monographs~\cite{pinkus} and~\cite{tnnbook} for extensive classical theory and applications.

In the mid 2000's, deep connections between the theory of TN matrices and the prime ideal theory of the so-called \emph{algebra of quantum matrices} were discovered (see~\cite{launois} for a survey). Therefore one might hope to bring tools developed in one field to bear on the other. Indeed Cauchon's Deleting Derivations Algorithm of ~\cite{cauchon} developed in the context of quantum algebra was used by Adm et al.~\cite{adm} to study the rank of TN matrices. In the other direction, work of Postnikov~\cite{postnikov} inspired the ``path model'' of quantum matrices in~\cite{casteels}.

This paper continues the theme by transferring the path model of quantum matrices back to the study of TP matrices and to what we will call the \emph{scaffolding} of a matrix. We demonstrate its utility by addressing the problem of inserting a new row or column into a given TP matrix while maintaining total positivity. This is the \emph{Line Insertion Problem}. Johnson and Smith~\cite{johnsonsmith} first proposed and solved this problem by finding an algorithm that always results in a valid line insertion between any given two rows or columns of a given TP matrix. In contrast, the present work will lead to Theorem~\ref{InsertionTheorem} that gives a characterization of the set of possible line insertions as the set of strongly positive solutions to a homogeneous system of linear equations. 

This paper is structured as follows. Section 3 introduces the concept of a scaffolding of a TP matrix $X$ whereby we interpret each entry of $X$ as a sum over paths in a certain grid-like directed graph. We show that every TP matrix has a unique scaffolding and discuss consequences of the Lindstr\"om-Gessel-Viennot Lemma as applied to scaffoldings. We then use these ideas in Section 4 to address the Line Insertion Problem and derive the aforementioned system of equations. Finally we show these equations do indeed have a strongly positive solution, thereby solving the Line Insertion Problem in a new way.

\section{Preliminaries}

For a positive integer $k$, let $[k]=\{1,2,\ldots,k\}$. Unless otherwise noted, the rows and and columns of an $m\times n$ matrix are indexed in the usual way by $[m]$ and $[n]$. We may also index rows and columns using other sets but we trust the reader will easily extend the material below to these situations.

If $A$ is an $m\times n$ matrix and $I$ and $J$ are subsets of the rows and columns of $A$ respectively, then $A[I, J]$ denotes the submatrix of $A$ formed by $I$ and $J$. When $|I|=|J|$, the determinant $\det A[I,J]$ is a \emph{minor} of $A$. 

A matrix is \emph{positive} if all entries are positive. A matrix is \emph{totally positive} (TP) if every minor of that matrix is positive. Of course every TP matrix is positive but not conversely.

Let $A$ be an $m\times n$ matrix. The notation $$A[\{i\cdots\}, \{j\cdots\}]$$ will be shorthand for the submatrix $$A[\{i,i+1,\ldots,i+k\}, \{j,j+1,\ldots, j+k\}]$$ where $k=\min(m-i,n-j).$ In other words it is the largest contiguous submatrix with $(i,j)$ in the top-left corner. Similarly, $$A[\{\cdots i\}, \{\cdots j\}]$$ is  the largest contiguous submatrix with $(i,j)$ in the bottom-right corner. We extend this notation to, for $i_0<i$ and $j_0<j$, setting $$A[\{i_0,i\cdots\}, \{j_0,j\cdots\}] = A[\{i_0,i,i+1,\ldots,i+k\}, \{j_0,j,j+1,\ldots, j+k\}]$$ where $k=\min(m-i,n-j).$ Similarly define $$A[\{\cdots i,i_0\}, \{\cdots j,j_0\}]$$ for $i_0>i$ and $j_0>j$.

It is easily seen that the transpose of a TP matrix is again TP. If $H_k$ denotes the $k\times k$ matrix with a $1$ in the entries $\{(1,k), (2,k-1),\ldots\}$ and $0$ everywhere else, then we also have the following which is immediate from Theorem 1.4.1 in~\cite{tnnbook}.

\begin{prop}\label{AntitransposeProp}
If $A$ is an $m\times n$ TP matrix, then $A^\tau=H_nA^TH_m$ is an $n\times m$ TP matrix.
	
\end{prop}

The matrix $A^\tau$ may be thought of as the reflection of $A$ across the ``anti-diagonal'' $\{(1,n), (2,n-1),\ldots\}$ and so we will call the map $A\mapsto A^\tau$ the \emph{anti-transpose}. 

Note that because of the transpose (or anti-transpose) map, it suffices to explain how to insert a row in order to solve the Line Insertion Problem.
 
It is convenient to use the language of directed graphs to visualize some of the concepts in this paper. We need nothing beyond the most elementary definitions here, however when we talk about \emph{paths} in a directed graph, we always will mean \emph{directed paths}. We also write $P\in G$ to mean the path $P$ is contained in the directed graph $G$. Additional special notation will be given in Notation~\ref{PathNotation}.

Finally, an $n$-tuple $\vect{v}\in\mathbb{R}^n$ is \emph{strongly positive} if each component of $\vect{v}$ is positive.

\section{Scaffoldings of Totally Positive Matrices}

\subsection{From $\Gamma$-Scaffoldings to Totally Positive Matrices}

The $\Gamma$-scaffolding of an $m\times n$ TP matrix $X$ is defined using a certain vertex-weighted directed graph. Roughly speaking, this graph is an $m\times n$ grid with extra vertices attached to the right and below, one for each row and column, and with horizontal edges oriented ``right to left'' and vertical edges ``top to bottom.'' Figures~\ref{ExampleFigure1} and ~\ref{ExampleFigure2} below are examples of such a graph. 
 
\begin{defn} Let $T=[t_{ij}]$ be an $m\times n$ positive matrix. Define the vertex-weighted directed graph $G^\Gamma_{m,n}(T)$ as follows. The vertex set is the disjoint union $([m]\times [n])\cup [m]\cup [n]$. The vertices $[m]$ are the \emph{row vertices} and the vertices $[n]$ are the \emph{column vertices.}\footnote{We resolve any ambiguity between these labels by explicitly stating the type (row or column) of vertex we mean.}

Next, for each $i\in [m]$, there is a directed edge from row vertex $i$ to the vertex $(i,n)$ and a directed edge from $(i,j)$ to $(i,j-1)$ for each $j\in [n]\setminus \{1\}$. These directed edges will be called \emph{horizontal edges}. Also, for each $j\in [n]$ and $i\in [m-1]$ there is a directed edge from $(i,j)$ to $(i+1,j)$ and a directed edge from $(m,j)$ to column vertex $j$. These directed edges will be called \emph{vertical edges}.

Finally, equip $G^\Gamma_{m,n}(T)$ with the function $w:[m]\times [n]\to \mathbb{R}$ defined by $w(i,j)=t_{ij}$.

\end{defn} 

In drawings of $G^\Gamma_{m,n}(T)$, we will label the internal vertices by their weight.

Suppose we have the graph $G_{m,n}^\Gamma(T)$ and let $P$ be a path in this graph that starts at row vertex $i$ and ends at column vertex $j$. Notice that $P$ is uniquely determined by the sequence of vertices at which it \emph{turns}: either proceeding from a horizontal edge to a vertical edge (\emph{$\Gamma$-turns}), or from a vertical edge to a horizontal edge ({\reflectbox{L}-\emph{turns}). In fact if $$((i,j_1), (i_2,j_1),\ldots, (i_\ell,j))$$ is this sequence of turns, then it alternates between $\Gamma$-turns and $\reflectbox{L}$-turns, starting and ending with a $\Gamma$-turn. Paths are crucial in this work so we here set some notation.

\begin{notn} \label{PathNotation}

With respect to the graph $G^\Gamma_{m,n}(T)$,
\begin{enumerate}
\item A path $P$ starting at vertex $v$ and ending at vertex $w$ will be denoted $P\colon v\to w$. 
\item Paths that start at an internal vertex $(a,b)$ and end at a column vertex will \emph{always be assumed to begin with a vertical edge}. 
\item Let $\ell\in [n]$ be a column index. If the path $P$ starts at row vertex $i$, ends at column vertex $j$ and has its first turn at a vertex $(i,j_1)$ for some $j_1\leq \ell$,  then write $P_{\leq \ell}\colon i\to j.$

\item Let $P\colon i\to j$ be a path with associated sequence of turns $$((i,j_1), (i_2,j_1),\ldots, (i_\ell,j)).$$ Define the \emph{weight} of $P$ to be $$w(P)=t_{ij_1}t_{i_2j_1}^{-1}t_{i_3j_3}\cdots t_{i_{\ell},j_{\ell-1}}^{-1}t_{i_{\ell}j}.$$

\item There exists a unique path $P\colon i\to j$ with exactly one $\Gamma$-turn, and weight $t_{ij}$. We call this path the \emph{primary path} from $i$ to $j$.
\end{enumerate}

\end{notn}

The following definition is crucial to this work.
\begin{defn}

Let $T=[t_{ij}]$ be an $m\times n$ positive matrix. Define the $m\times n$ matrix $X(T)=[x_{ij}]$ by $$x_{ij} = \sum_{\substack{P\colon i\to j,\\P\in G_{m,n}^\Gamma(T)}}w(P).$$ 

We say that $T$ is the \emph{$\Gamma$-scaffolding} of $X(T)$.
\end{defn}

\begin{ex}\label{example1}

Let $$T=\begin{bmatrix} 1 & 3 & 1 \\ 1 & \frac{1}{2} & 1\end{bmatrix}.$$ Then $G^\Gamma_{2,3}(T)$ is illustrated in Figure~\ref{ExampleFigure1} and $$X(T) = \begin{bmatrix} 1 + 3(\frac{1}{2})^{-1}1+1(1)^{-1}1 & 3+1(1)^{-1}\frac{1}{2} & 1 \\ 1 & \frac{1}{2} & 1 \end{bmatrix} = \begin{bmatrix} 8 & \frac{7}{2} & 1\\ 1 & \frac{1}{2} & 1\end{bmatrix}.$$ Notice that $X(T)$ is TP.
 \begin{figure}[ht]
\begin{tikzpicture}[xscale=1.2, yscale=1.2]

\node at (0,0) {$\bullet$};
\node at (0,-0.35) {$1$};

\node at (1,0) {$\bullet$};
\node at (1,-0.35) {$2$};

\node at (2,0) {$\bullet$};
\node at (2,-0.35) {$3$};

\node at (3,2) {$\bullet$};
\node at (3.35,2) {$1$};

\node at (3,1) {$\bullet$};
\node at (3.35,1) {$2$};

\node at (0,1) {$\bullet$};
\node at (-0.15,1.2) {$1$};
\node at (1,1) {$\bullet$};
\node at (0.85,1.3) {$\frac{1}{2}$};
\node at (2,1) {$\bullet$};
\node at (1.85,1.2) {$1$};

\node at (0,2) {$\bullet$};
\node at (-0.15,2.2) {$1$};
\node at (1,2) {$\bullet$};
\node at (0.85,2.2) {$3$};
\node at (2,2) {$\bullet$};
\node at (1.85,2.2) {$1$};

\draw [->, thick, black] (3,2) -- (2.1,2);
\draw [->, thick, black] (2,2) -- (1.1,2);
\draw [->, thick, black] (1,2) -- (0.1,2);


\draw [->, thick, black] (3,1) -- (2.1,1);
\draw [->, thick, black] (2,1) -- (1.1,1);
\draw [->, thick, black] (1,1) -- (0.1,1);

\draw [->, thick, black] (0,2) -- (0,1.1);
\draw [->, thick, black] (0,1) -- (0,0.1);

\draw [->, thick, black] (1,2) -- (1,1.1);
\draw [->, thick, black] (1,1) -- (1,0.1);

\draw [->, thick, black] (2,2) -- (2,1.1);
\draw [->, thick, black] (2,1) -- (2,0.1);

\end{tikzpicture}

\caption{The graph $G^\Gamma_{2,3}(T)$ of Example~\ref{example1}. Internal vertices are labeled by their weights.} \label{ExampleFigure1}
	\end{figure}
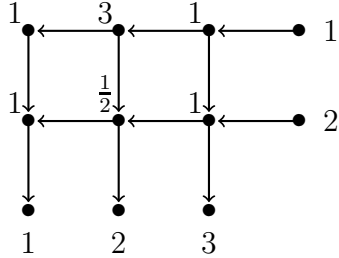

\end{ex}

\begin{ex}\label{3x3Example2}

Let $T=[t_{ij}]$ be a positive $3\times 3$ matrix. Then $G^\Gamma_{3,3}(T)$ is illustrated in Figure~\ref{ExampleFigure2} and $X(T)=[x_{ij}]$ is the $3\times 3$ matrix with
\begin{align*} x_{11} & =
 t_{11}+t_{12}t_{22}^{-1}t_{21} + t_{12}t_{32}^{-1}t_{31} + t_{13}t_{23}^{-1}t_{21}+ t_{13}t_{23}^{-1}t_{22}t_{32}^{-1}t_{31}+t_{13}t_{33}^{-1}t_{31},\\
 x_{12} &= t_{12}+t_{13}t_{23}^{-1}t_{22} + t_{13}t_{33}^{-1}t_{32},\\
 x_{21}& = t_{21}+t_{22}t_{32}^{-1}t_{31} + t_{23}t_{33}^{-1}t_{31},\\
 x_{22}&=t_{22}+t_{23}t_{33}^{-1}t_{32},\\
 x_{ij} &= t_{ij}, \textrm{\quad if $i=3$ or $j=3$.}
 \end{align*}
 
  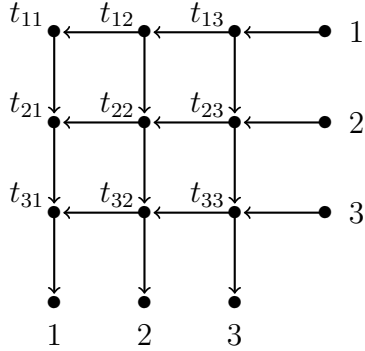
\begin{figure}[ht]
\begin{tikzpicture}[xscale=1.2, yscale=1.2]

\node at (0,0) {$\bullet$};
\node at (0,-0.35) {$1$};

\node at (1,0) {$\bullet$};
\node at (1,-0.35) {$2$};

\node at (2,0) {$\bullet$};
\node at (2,-0.35) {$3$};

\node at (3,3) {$\bullet$};
\node at (3.35,3) {$1$};

\node at (3,2) {$\bullet$};
\node at (3.35,2) {$2$};

\node at (3,1) {$\bullet$};
\node at (3.35,1) {$3$};

\node at (0,1) {$\bullet$};
\node at (-0.3,1.2) {$t_{31}$};
\node at (1,1) {$\bullet$};
\node at (0.7,1.2) {$t_{32}$};
\node at (2,1) {$\bullet$};
\node at (1.7,1.2) {$t_{33}$};

\node at (0,2) {$\bullet$};
\node at (-0.3,2.2) {$t_{21}$};
\node at (1,2) {$\bullet$};
\node at (0.7,2.2) {$t_{22}$};
\node at (2,2) {$\bullet$};
\node at (1.7,2.2) {$t_{23}$};

\node at (0,3) {$\bullet$};
\node at (-0.3,3.2) {$t_{11}$};
\node at (1,3) {$\bullet$};
\node at (0.7,3.2) {$t_{12}$};
\node at (2,3) {$\bullet$};
\node at (1.7,3.2) {$t_{13}$};

\draw [->, thick, black] (3,2) -- (2.1,2);
\draw [->, thick, black] (2,2) -- (1.1,2);
\draw [->, thick, black] (1,2) -- (0.1,2);

\draw [->, thick, black] (3,3) -- (2.1,3);
\draw [->, thick, black] (2,3) -- (1.1,3);
\draw [->, thick, black] (1,3) -- (0.1,3);

\draw [->, thick, black] (3,1) -- (2.1,1);
\draw [->, thick, black] (2,1) -- (1.1,1);
\draw [->, thick, black] (1,1) -- (0.1,1);

\draw [->, thick, black] (0,3) -- (0,2.1);
\draw [->, thick, black] (0,2) -- (0,1.1);
\draw [->, thick, black] (0,1) -- (0,0.1);

\draw [->, thick, black] (1,3) -- (1,2.1);
\draw [->, thick, black] (1,2) -- (1,1.1);
\draw [->, thick, black] (1,1) -- (1,0.1);

\draw [->, thick, black] (2,3) -- (2,2.1);
\draw [->, thick, black] (2,2) -- (2,1.1);
\draw [->, thick, black] (2,1) -- (2,0.1);

\end{tikzpicture}

\caption{The graph $G^\Gamma_{3,3}(T)$.} \label{ExampleFigure2}
	\end{figure}

\end{ex}

\subsection{Minors and $\Gamma$-scaffolding}

When $T$ is a positive matrix, it turns out that $X(T)$ is totally positive. To see why this is, we need a relationship between minors of $X(T)$ and the $\Gamma$-scaffolding $T$. This is provided by the well-known Lindst\"om-Gessel-Viennot Lemma.

To explain, fix a $G^\Gamma_{m,n}(T)$, and let $I=\{i_1<i_2<\cdots <i_k\}$ be a subset of the row vertices and $J=\{j_1<\cdots < j_k\}$ a subset of column vertices with $|I|=|J|$. A \emph{path system from $I$ to $J$} in $G^\Gamma_{m,n}(T)$ is a sequence $\C{P}=(P_1,P_2,\ldots,P_k)$ of paths where $P_\ell\colon i_\ell\to j_\ell$ for each $\ell\in[k]$. We say that $\C{P}$ is \emph{vertex-disjoint} if its paths are mutually vertex-disjoint. Finally, the \emph{weight} of the path system is the product of the weights of its paths, i.e., $$w(\C{P})=w(P_1)w(P_2)\cdots w(P_k).$$ Note that since $T$ is positive, so is $w(\C{P})$.

We may now state the following special case of the Lindst\"om-Gessel-Viennot Lemma (see ~\cite{gesselviennot}).

\begin{lem} \label{lindstrom}
Let $T$ be a positive $m\times n$ matrix and set $X=X(T)$. If $I\subseteq [m]$ and $J\subseteq [n]$ are such that $|I|=|J|$, then $$\det X[I,J] = \sum_{\C{P}}w(\C{P}),$$ where the sum is over all vertex-disjoint path systems from $I$ to $J$ in $G^\Gamma_{m,n}(T).$
\end{lem}

It should be noted that the Lindstr\"om-Gessel-Viennot Lemma is usually stated for edge-weighted directed graphs. The graph $G^{\Gamma}_{m,n}(T)$ can be modified to this setting by defining the edge-weight of all vertical edges to be $1$, the weight of the edge from the row vertex $i$ to $(i,n)$ to be $t_{i,n}$, and the weight of the edge from $(i,j)$ to $(i,j-1)$ to be $t_{i,j-1}t_{i,j}^{-1}.$ With this scheme, it is easy to verify that the edge-weight of a path (being the product of the edge weights) equals the (vertex-) weight of a path as defined in Notation~\ref{PathNotation}.

\begin{ex} 

Referring back to Figure~\ref{3x3Example2} where $T=[t_{ij}]$ is a $3\times 3$ positive matrix and $X=X(T)$, one has, for example, that \begin{align*} \det X[\{1,2\}, \{1,2\}] &= t_{11}\cdot t_{22}+ t_{11}\cdot t_{23}t_{33}^{-1}t_{32} + t_{12}t_{22}^{-1}t_{21}\cdot t_{23}t_{33}^{-1}t_{32}\\
 \det X[\{1,2\}, \{1,3\}] &= t_{11}\cdot t_{23} + t_{12}t_{22}^{-1}t_{21}\cdot t_{23} + t_{12}t_{32}^{-1}t_{31}\cdot t_{23},\end{align*} and $$\det(X) = t_{11}t_{22}t_{33}.$$

\end{ex}

Note that there always exists at least one vertex-disjoint path system from $I$ to $J$ in $G^\Gamma_{m,n}(T)$, namely $\C{P}=(P_1,\ldots,P_k)$ where each $P_\ell$ is the primary path from $i_\ell$ to $j_\ell$. Call this the \emph{primary path system from $I$ to $J$}. The existence of this path system together with Lemma~\ref{lindstrom} has two immediate consequences. The first keeps our earlier promise.

\begin{cor} \label{LindstromCor1}
If $T$ is a positive matrix, then $X(T)$ is totally positive.
	
\end{cor}

The second corollary is related to certain contiguous minors and will be needed for our work in Section~\ref{InsertionSection}. First, notice that if $$X[I,J]=X[\{i\cdots\}, \{j\cdots\} ]$$ is a contiguous submatrix of $X(T)$, then the primary path system from $I$ to $J$ in $G_{m,n}^\Gamma(T)$ is in fact the \emph{unique} path system from $I$ to $J$. It follows that $$\det X[\{i\cdots\},\{j\cdots\}] = t_{i,j}t_{i+1,j+1}\cdots t_{i+k,j+k},$$ where $k=\min(m-i,n-j).$

In Section~\ref{InsertionSection} we will encounter sums of the form $$\sum_{P_{\leq \ell}\colon i\to j} w(P).$$ We may use Lemma~\ref{lindstrom} to write this quantity in terms of minors of $X(T)$. The paths in this sum may be thought of as exactly those paths from row vertex $i$ to column vertex $j$ that are ``blocked'' by (i.e., are disjoint from) the paths in the primary path system from $\{i+1\cdots\}$ to $\{\ell+1\cdots\}$. See Figure~\ref{VisualProofFigure2}. Given this, the next result follows immediately from Lemma~\ref{lindstrom}.

  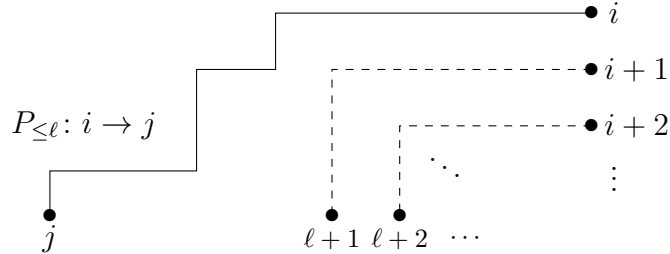
\begin{figure}[ht]
\begin{tikzpicture}[xscale=1.5, yscale=1.5]

\node at (11,2) {$i$};\
\node at (10.8,2) {$\bullet$};
\node at (6.3,1) {$P_{\leq \ell}\colon i\to j$};
\node at (6,0) {$j$};
\node at (6,0.2) {$\bullet$};

\draw (10.8,2) -- (8,2) -- (8,1.5)--(7.3,1.5)--(7.3,0.6)--(6,0.6)--(6,0.2) ;

\draw[dashed] (10.8,1.5)--(8.5,1.5) -- (8.5,0.2);
\draw[dashed] (10.8,1)--(9.1,1) -- (9.1,0.2);
\node at (11.2,1.5) {$i+1$};
\node at (10.8,1.5) {$\bullet$};
\node at (10.8,1) {$\bullet$};
\node at (11.2,1) {$i+2$};
\node at (11,0.6) {$\vdots$};
\node at (9.1,0.2) {$\bullet$};
\node at (8.5,0.2) {$\bullet$};
\node at (8.5,0) {\Small $\ell+1$};
\node at (9.1,0) {\Small $\ell+2$};
\node at (9.7,0) {\Small $\cdots$};

\node at (9.5,0.7) {$\ddots$};
\end{tikzpicture}

\caption{A path $P_{\leq \ell}\colon i\to j$ that is ``blocked'' by the primary path system (dashed) from $\{i+1\cdots\}$ to $\{\ell+1\cdots\}$.} \label{VisualProofFigure2}
	\end{figure}

\begin{cor} \label{LindstromCor2}
Let $T=[t_{ij}]$ be a positive matrix and set $X=X(T)$. Let $i$ be a row vertex and $j,\ell\in [n]$ column vertices with $j\leq \ell$. Then $$\sum_{P_{\leq \ell}\colon i\to j} w(P) = \frac{\det X[\{i,i+1\cdots\},\{j,\ell+1\cdots\}]}{\det X[\{i+1\cdots\},\{ \ell+1\cdots\}]}.$$

\end{cor}

It may be worth pointing out that setting $\ell=j$  gives the following formula for the entries of the $\Gamma$-scaffolding of $X=X(T)$ in terms of minors of $X$: $$t_{ij}=\frac{\det X[\{i\cdots\},\{j\cdots\}]}{\det X[\{i+1\cdots\},\{j+1\cdots\}]}.$$ 

\subsection{From Totally Positive Matrices to $\Gamma$-Scaffoldings}

Corollary~\ref{LindstromCor1} begs the question: does every totally positive matrix $X$ have a $\Gamma$-scaffolding $T$, i.e., a positive matrix $T$ with $X=X(T)$? The answer is yes.

One may find the $\Gamma$-scaffolding of $X$ using the next procedure. To explain, totally order $[m]\times [n]$ using the reverse lexicographic order $\prec$, that is, $(i,j)\prec (k,\ell)$ if either $i>k$, or $i=k$ and $j>\ell$. If $(i,j)\in [m]\times [n]$, then set $(i,j)^+$ to be the next largest element in this order.

\begin{algorithm}[Cauchon's  Algorithm \cite{cauchon}]
Let $X$ be an $m\times n$ TP matrix. 

\begin{enumerate}

\item Set $X^{(m,n)}=X$. 
\item Suppose $X^{(i,j)}=[x^{(i,j)}_{k\ell}]$ has been defined.
If $(i,j)\prec (1,1)$, then
set $$x_{k\ell}^{(i,j)^+} = \begin{cases} x_{k,\ell}^{(i,j)}-x_{kj}^{(i,j)}\left(x_{ij}^{(i,j)}\right)^{-1}x_{i\ell}^{(i,j)},  & \textrm{if $k<i$ and $\ell<j$,} \\ x_{k\ell}^{(i,j)}, & \textrm{ otherwise.}\end{cases}$$

\item Set $T=X^{(1,1)}$.

\end{enumerate}
\end{algorithm}

\begin{ex}

For $$X = X^{(2,3)}=\begin{bmatrix} 8 & \frac{7}{2} & 1\\ 1 & \frac{1}{2} & 1\end{bmatrix},$$ there are effectively only two steps in Cauchon's Algorithm. The first step results in $$X^{(2,2)} = \begin{bmatrix} 8 - 1(1)^{-1}1 & \frac{7}{2} - 1( 1)^{-1}\frac{1}{2}& 1\\ 1 & \frac{1}{2} & 1\end{bmatrix} = \begin{bmatrix} 7 & 3 & 1\\ 1 & \frac{1}{2} & 1\end{bmatrix}.$$
Then $$X^{(2,1)} = \begin{bmatrix} 7-3\left(\frac{1}{2}\right)^{-1} 1 & 3 & 1\\ 1 & \frac{1}{2} & 1\end{bmatrix} = \begin{bmatrix} 1& 3 & 1\\ 1 & \frac{1}{2} & 1\end{bmatrix}.$$

The remainder of the steps do not change this matrix and so the output $T=X^{(2,1)}$. Notice this is the same $T$ that began Example~\ref{example1}.

\end{ex}

In general, if $i=1$ or $j=1$, then one has $X^{(i,j)^+}=X^{(i,j)}$ and so these steps may be skipped. It follows that for $m\geq 2$, $T=X^{(2,1)}$. 

Of course Cauchon's Algorithm is not \emph{a priori} sensible since the $(i,j)$-entry of $X^{(i,j)}$ could have ended up as zero. Fortunately this never happens, and in fact we have the following. 

\begin{thm}[\cite{adm2}, Theorem 3.3] \label{CauchonTheorem}
If $X$ is a TP matrix, then every step of Cauchon's Algorithm produces a positive matrix.  Moreover, the entries of $X^{(i,j)}$ in positions $\{(1,1)\succ(1,2)\succ\cdots\succ (i,j)\}$ form a \emph{partial} TP matrix, i.e., all minors that are completely determined by these coordinates are positive.
\end{thm}

It may be helpful to understand Cauchon's Algorithm as simply a careful reversal of the process of forming $X(T)=[x_{ij}]$ from $T$. Indeed all paths from $k$ to $\ell$ other than the primary path contain at least one $\reflectbox{L}$-turn, and therefore a final $\reflectbox{L}$-turn. Thus we can decompose each sum $x_{k\ell}$ as $$x_{k\ell}=t_{k\ell}+\sum_{(i,j)\in [m]\times [n]} \left(\sum_{P} w(P)\right),$$ where each interior sum is over those paths from row vertex $k$ to column vertex $\ell$ whose last \reflectbox{L}-turn occurs at vertex $(i,j)$. (Many of these interior sums may be 0.) Then a careful analysis reveals that the $(i,j)$-step of Cauchon's Algorithm is \emph{deleting} from each $x_{k\ell}$ the set of paths in $G_{m,n}^\Gamma(T)$ from row vertex $k$ to column vertex $\ell$ whose last $\reflectbox{L}$-turn occurs at $(i,j)$. 

It follows that the intermediate matrices $X^{(i,j)}=[x_{k\ell}^{(i,j)}]$ in Cauchon's Algorithm can be formed from the final output $T=[t_{ij}]$ by $x_{k\ell}^{(i,j)} = \sum_{P\colon k\to \ell}w(P)$ where the sum is over all paths whose \reflectbox{L}-turns occur only at vertices greater than or equal to $(i,j)$ in the reverse lexicographic order. Hence we can state the following.

\begin{cor}
If $X$ is a TP matrix and Cauchon's Algorithm applied to $X$ results in the positive matrix $T$, then $T$ is the $\Gamma$-scaffolding of $X$.
\end{cor}

\subsection{\reflectbox{L}-scaffoldings.}

Recall that if $X$ is a TP matrix, then so is its anti-transpose $X^\tau$. Thus $X^\tau$ has a $\Gamma$-scaffolding, say $S$ which may be found by Cauchon's Algorithm. The matrix $T=S^\tau$ will be called the \reflectbox{L}-\emph{scaffolding} of $X$. 

We can modify the concepts above to avoid the intermediate use of the anti-transpose. Let us summarize. For a positive matrix $T$, the directed graph $G^{\reflectbox{\SMALL L}}_{m,n}(T)$ is, roughly speaking, the ``anti-transpose'' of  $G_{n,m}^\Gamma(S)$ with all edge orientations reversed. See Figures~\ref{Le3x3ExampleFigure2} and~\ref{LeScaffoldingExampleFigure} below for examples.

A path $P\colon i\to j$ in $G^{\reflectbox{\SMALL L}}_{m,n}(T)$ is still determined by its alternating sequence of $\reflectbox{L}$-turns and $\Gamma$-turns, but now this sequence begins and ends with $\reflectbox{L}$-turns. Still, if $((i,j_1), (i_2,j_1),\ldots, (i_\ell,j))$ is this sequence of turns, then (as before) define $$w(P)=t_{ij_1}t_{i_2j_1}^{-1}t_{i_3j_3}\cdots t_{i_{\ell},j_{\ell-1}}^{-1}t_{i_{\ell}j}.$$

Now define $X(T)=[x_{ij}]$ by $$x_{ij}=\sum_{\substack{P\colon i\to j,\\P\in G^{\reflectbox{\SMALL L}}_{m,n}(T)}} w(P).$$  Call $T$ the \reflectbox{L}-\emph{scaffolding} of $X(T)$.

\begin{ex}\label{3x3Example2}

Let $T=[t_{ij}]$ be a positive $3\times 3$ matrix. Then $G^{\reflectbox{\SMALL L}}_{3,3}(T)$ is illustrated in Figure~\ref{Le3x3ExampleFigure2} and $X(T)=[x_{ij}]$ is the $3\times 3$ matrix with
\begin{align*} x_{33} & =
 t_{33}+t_{32}t_{22}^{-1}t_{23} + t_{32}t_{12}^{-1}t_{13} + t_{31}t_{21}^{-1}t_{23}+ t_{31}t_{21}^{-1}t_{22}t_{12}^{-1}t_{13}+t_{31}t_{11}^{-1}t_{13},\\
 x_{32} &= t_{32}+t_{31}t_{21}^{-1}t_{22} + t_{31}t_{11}^{-1}t_{12},\\
 x_{23}& = t_{23}+t_{22}t_{12}^{-1}t_{13} + t_{21}t_{11}^{-1}t_{13},\\
 x_{22}&=t_{22}+t_{21}t_{11}^{-1}t_{12},\\
 x_{ij} &= t_{ij}, \textrm{\quad if $i=1$ or $j=1$.}
 \end{align*}
 
  \begin{figure}[ht]
\begin{tikzpicture}[xscale=1.2, yscale=1.2]

\node at (0,4) {$\bullet$};
\node at (0,4.3) {$1$};

\node at (1,4) {$\bullet$};
\node at (1,4.3) {$2$};

\node at (2,4) {$\bullet$};
\node at (2,4.3) {$3$};

\node at (-1,3) {$\bullet$};
\node at (-1.3,3) {$1$};

\node at (-1,2) {$\bullet$};
\node at (-1.3,2) {$2$};

\node at (-1,1) {$\bullet$};
\node at (-1.3,1) {$3$};

\node at (0,1) {$\bullet$};
\node at (-0.3,1.2) {$t_{31}$};
\node at (1,1) {$\bullet$};
\node at (0.7,1.2) {$t_{32}$};
\node at (2,1) {$\bullet$};
\node at (1.7,1.2) {$t_{33}$};

\node at (0,2) {$\bullet$};
\node at (-0.3,2.2) {$t_{21}$};
\node at (1,2) {$\bullet$};
\node at (0.7,2.2) {$t_{22}$};
\node at (2,2) {$\bullet$};
\node at (1.7,2.2) {$t_{23}$};

\node at (0,3) {$\bullet$};
\node at (-0.3,3.2) {$t_{11}$};
\node at (1,3) {$\bullet$};
\node at (0.7,3.2) {$t_{12}$};
\node at (2,3) {$\bullet$};
\node at (1.7,3.2) {$t_{13}$};

\draw [->, thick, black] (-1,2) -- (-0.1,2);
\draw [<-, thick, black] (1.9,2) -- (1,2);
\draw [<-, thick, black] (0.9,2) -- (0,2);

\draw [->, thick, black] (-1,3) -- (-0.1,3);
\draw [<-, thick, black] (1.9,3) -- (1,3);
\draw [<-, thick, black] (0.9,3) -- (0,3);

\draw [->, thick, black] (-1,1) -- (-0.1,1);
\draw [<-, thick, black] (1.9,1) -- (1,1);
\draw [<-, thick, black] (0.9,1) -- (0,1);

\draw [<-, thick, black] (0,2.9) -- (0,2);
\draw [<-, thick, black] (0,1.9) -- (0,1);
\draw [->, thick, black] (0,3) -- (0,3.9);

\draw [<-, thick, black] (1,2.9) -- (1,2);
\draw [<-, thick, black] (1,1.9) -- (1,1);

\draw [->, thick, black] (1,3) -- (1,3.9);
\draw [<-, thick, black] (2,2.9) -- (2,2);
\draw [<-, thick, black] (2,1.9) -- (2,1);

\draw [->, thick, black] (2,3) -- (2,3.9);

\end{tikzpicture}

\caption{The graph $G^{\reflectbox{\SMALL L}}_{3,3}(T)$.} \label{Le3x3ExampleFigure2}
	\end{figure}

\end{ex}

To find the \reflectbox{L}-scaffolding $T$ of a TP matrix $X$, Cauchon's Algorithm proceeds using the ordering $\prec_0$ $[m]\times [n]$ defined by $(i,j)\prec_0 (k,\ell)$ if $j<\ell$, or $j=\ell$ and $i<k$. 

\begin{algorithm}[Cauchon's Algorithm (\reflectbox{L}-version)]

Let $X$ be an $m\times n$ TP matrix. 

\begin{enumerate}

\item Set $X^{(1,1)}=X$. 
\item Suppose $X^{(i,j)}=[x^{(i,j)}_{k\ell}]$ has been defined.
If $(i,j)\prec (m,n)$, then
set $$x_{k\ell}^{(i,j)^+} = \begin{cases} x_{k,\ell}^{(i,j)}-x_{kj}^{(i,j)}\left(x_{ij}^{(i,j)}\right)^{-1}x_{i\ell}^{(i,j)},  & \textrm{if $k>i$ and $\ell>j$,} \\ x_{k\ell}^{(i,j)}, & \textrm{ otherwise.}\end{cases}$$

\item Set $T=X^{(m,n)}$.
\end{enumerate}
\end{algorithm}

\begin{ex} \label{Leexample1}

Once again, let 
$$X =X^{(1,1)}=\begin{bmatrix} 8 & \frac{7}{2} & 1\\ 1 & \frac{1}{2} & 1\end{bmatrix}.$$
Then we have $$X^{(2,1)} = \begin{bmatrix} 8 & \frac{7}{2} & 1\\ 1 & \frac{1}{2} - 1(8)^{-1}\frac{7}{2} & 1-1(8)^{-1}1\end{bmatrix} = \begin{bmatrix} 8 & \frac{7}{2} & 1\\ 1 & \frac{1}{16} & \frac{7}{8}\end{bmatrix}$$ and $$X^{(1,2)}=  \begin{bmatrix} 8 & \frac{7}{2} & 1\\ 1 & \frac{1}{16} & \frac{7}{8} - 1\left(\frac{7}{2}\right)^{-1}\frac{1}{16}\end{bmatrix} = \begin{bmatrix} 8 & \frac{7}{2} & 1\\ 1 & \frac{1}{16} & \frac{6}{7}\end{bmatrix}.$$

As in the $\Gamma$-scaffolding case, this is effectively the final step of Cauchon's Algorithm. The graph $G^{\reflectbox{\SMALL L}}_{m,n}(T)$ is in Figure~\ref{LeScaffoldingExampleFigure}. Notice, for example, that $$\sum_{P\colon 2\to 3} w(P) = \frac{6}{7} + \frac{1}{16} \left(\frac{7}{2}\right)^{-1}1 + 1(8)^{-1} 1 = 1 = x_{23}.$$

 \begin{figure}[ht]
\begin{tikzpicture}[xscale=1.3, yscale=1.3]

\node at (0,3) {$\bullet$};
\node at (0,3.35) {$1$};

\node at (1,3) {$\bullet$};
\node at (1,3.35) {$2$};

\node at (2,3) {$\bullet$};
\node at (2,3.35) {$3$};

\node at (-1,2) {$\bullet$};
\node at (-1.35,2) {$1$};

\node at (-1,1) {$\bullet$};
\node at (-1.35,1) {$2$};

\node at (0,1) {$\bullet$};
\node at (-0.15,1.2) {$1$};
\node at (1,1) {$\bullet$};
\node at (0.85,1.3) {$\frac{1}{16}$};
\node at (2,1) {$\bullet$};
\node at (1.85,1.3) {$\frac{6}{7}$};

\node at (0,2) {$\bullet$};
\node at (-0.15,2.3) {$8$};
\node at (1,2) {$\bullet$};
\node at (0.85,2.3) {$\frac{7}{2}$};
\node at (2,2) {$\bullet$};
\node at (1.85,2.3) {$1$};

\draw [->, thick, black] (-1,2) -- (-0.1,2);
\draw [<-, thick, black] (1.9,2) -- (1,2);
\draw [<-, thick, black] (0.9,2) -- (0,2);

\draw [->, thick, black] (-1,1) -- (-0.1,1);
\draw [<-, thick, black] (1.9,1) -- (1,1);
\draw [<-, thick, black] (0.9,1) -- (0,1);

\draw [<-, thick, black] (0,2.9) -- (0,2);
\draw [<-, thick, black] (0,1.9) -- (0,1);

\draw [<-, thick, black] (1,2.9) -- (1,2);
\draw [<-, thick, black] (1,1.9) -- (1,1);

\draw [<-, thick, black] (2,2.9) -- (2,2);
\draw [<-, thick, black] (2,1.9) -- (2,1);

\end{tikzpicture}

\caption{The graph $G^{\reflectbox{\SMALL L}}_{2,3}(T)$ of Example~\ref{Leexample1}.} \label{LeScaffoldingExampleFigure}
	\end{figure}

\end{ex}

Finally, note that the obvious analogies of Lemma~\ref{lindstrom} and Corollaries~\ref{LindstromCor1} and~\ref{LindstromCor2}  in the setting of \reflectbox{L}-scaffoldings hold true.

\section{Line Insertion in Totally Positive Matrices} \label{InsertionSection}

\subsection{Bordering}
In this section, we reduce the TP line insertion problem to that of finding a strongly positive solution to a certain homogeneous system of linear equations. The key steps in our reduction use the idea of \emph{bordering} a TP matrix, i.e., adding a row or column to the \emph{outside} while retaining total positivity. That one may do this at all is well-known and easy to show. The novelty here is that we use scaffoldings to characterize all possible borderings.

Let $X$ be a TP matrix and suppose we wish to append a new row, with index $0$ say, above $X$. Do this as follows:
\begin{enumerate} \item Find the $\Gamma$-scaffolding $T$ of $X$. 
\item Append above $T$ a new row with only positive entries to form $T^\prime$.
\item The matrix $X(T^\prime)$ is TP and contains $X$ in rows $1,2,\ldots,m$.
\end{enumerate}

This method works since in $G_{m+1,n}^\Gamma(T^\prime)$, no path from a row vertex $i>0$ to column vertex $j$ turns in row $0$. On the other hand, every such path corresponds to a path in $G_{m,n}^\Gamma(T)$. Hence if $X^\prime=X(T^\prime)$, then $X$ is the submatrix of $X^\prime$ formed by rows $1,2,\ldots, m$. Conversely, suppose $X^\prime$ is an $(m+1)\times n$ TP matrix with rows indexed by $0,1,\ldots, n$ such that rows $1$ to $m$ form $X$. When applying Cauchon's Algorithm to $X^\prime$, the steps involving only entries in rows $1$ to $m$ are identical to the application of Cauchon's Algorithm to $X$. Hence the $\Gamma$-scaffolding $T^\prime$ of $X^\prime$ contains the $\Gamma$-scaffolding $T$ of $X$ in rows $1$ to $m$ together with a positive row $0$.

Next, to add a column to the left of $X$, we need only append a strongly positive column to the left of $T$. On the other hand, to append a row beneath $X$ or to the right of $X$, we proceed similarly but using the \reflectbox{L}-scaffolding of $X$ instead.

We now carefully analyze the output of this bordering technique in the case that we are adding a row above $X$. Suppose $T^\prime$ has been formed by appending $\begin{bmatrix} r_1 & r_2 & \cdots & r_n\end{bmatrix}$ above the $\Gamma$-scaffolding $T=[t_{ij}]$ of $X$ to form row $0$.  Let $x_{0j}$ be the $j$th entry in the $0$th row  of $X(T^\prime).$ By definition, $x_{0j}$ is the sum over all paths from row vertex $0$ to column vertex $j$ in $G^\Gamma_{m+1,n}(T^\prime)$. Each path begins with a $\Gamma$-turn at $(0,\ell)$ for some $j\leq \ell\leq n$ and (recalling Notation~\ref{PathNotation}) we conclude \begin{align} x_{0j} &= \sum_{\ell=j}^n \sum_{P:(0,\ell)\to j} w(P)r_\ell.\label{BorderingEquation}\end{align}

The coefficient of $r_\ell$ in Equation~(\ref{BorderingEquation}) may be written using minors of $X$. The key is to notice that if $P\colon (0,\ell)\to j$ is a path in $G_{m+1,1}^\Gamma(T^\prime)$, then $t_{1\ell}w(P)$ is the weight of a path $Q\colon 1\to j$. Moreover, $Q$ contains no turns in any columns from $\ell+1$ to $n$. See Figure~\ref{PvsQFigure}. Conversely, any path $Q: 1\to j$ with no turns in columns $\ell+1$ to $n$ arises in this way.
 \begin{figure}[ht]
\begin{tikzpicture}[xscale=1.3, yscale=1.3]

%
%

\node at (10,2) {$0$};\
\node at (9.8,2) {$\bullet$};
\node at (8,2.3) {$(0,\ell)$};
\node at (7.7,1.5) {$t_{1\ell}$};
\node at (8,2) {$\bullet$};
\node at (6,-0.1) {$j$};
\node at (6.05,0.2) {$\bullet$};
\node at (8,1.5) {$\bullet$};

\draw (8,2) -- (8,1)--(7.3,1)--(7.3,0.6)--(6,0.6)--(6,0.2) ;
\draw [dashed] (9.8,1.5) -- (8.1,1.5) -- (8.1,0.9)--(7.4,0.9)--(7.4,0.5)--(6.1,0.5)--(6.1,0.2) ;

\node at (10,1.5) {$1$};
\node at (9.8,1.5) {$\bullet$};
\node at (9.8,1) {$\bullet$};
\node at (10,1) {$2$};
\node at (10,0.5) {$\vdots$};

\end{tikzpicture}

\caption{Example of a $P\colon (0,\ell) \to  j$ (solid) and the corresponding $Q\colon 1\to j$ (dashed) in $G_{m+1,n}^\Gamma(T^\prime)$ where $w(Q)=t_{1\ell}w(P)$.} \label{PvsQFigure}
	\end{figure}
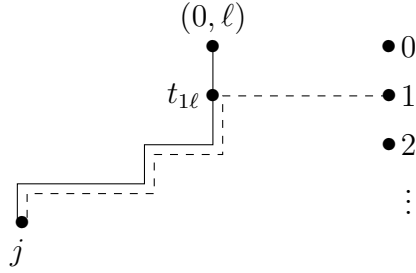

Therefore, $$\sum_{P:(0,\ell)\to j} w(P) = t_{1\ell}^{-1}\left(\sum_{P:(0,\ell)\to j}t_{1\ell} w(P)\right) = t_{1\ell}^{-1}\left(\sum_{Q_{\leq \ell}:1\to j} w(Q)\right).$$ Applying Corollary~\ref{lindstrom} gives us the following.

\begin{thm} \label{BorderingTheorem1}

Let $X$ be an $m\times n$ TP matrix with $\Gamma$-scaffolding $T$. Suppose $X^\prime$ is an $(m+1)\times n$ TP matrix obtained from $X$ by adding the new row $\begin{bmatrix}x_{01} & x_{02} & \cdots & x_{0n}\end{bmatrix}$ above the first. Then there exists a strongly positive $\vect{r}=\begin{bmatrix} r_1 & \cdots & r_n\end{bmatrix}$ such that $T^\prime = \begin{bmatrix} \vect{r} \\T\end{bmatrix}$ is the $\Gamma$-scaffolding of $X^\prime$ and for all $j\in [n]$, \begin{align} x_{0j} &= \sum_{\ell=j}^n \sum_{\substack{P:(0,\ell)\to j, \\P\in G_{m+1,n}^\Gamma(T^\prime)}} w(P)r_\ell \nonumber \\ &=\sum_{\ell=j}^n \frac{\det X[\{12\cdots\},\{j,\ell+1\cdots\}]}{\det X[\{12\cdots\},\{ \ell,\ell+1\cdots\}]}r_\ell.\label{BorderingTheoremEquation}\end{align} 

Conversely, if we take positive real numbers $r_1,r_2,\ldots,r_n$ and define $x_{0j}$ as in Equation~(\ref{BorderingTheoremEquation}), then the matrix $X^\prime$ obtained from $X$ by adding $$\begin{bmatrix}x_{01} & x_{02} & \cdots & x_{0n}\end{bmatrix}$$ above the first row is totally positive. 
	
\end{thm}

\begin{ex}

Let $T=\begin{bmatrix} 2 & 1 & 1\\ 1&  1 & 1\end{bmatrix}$ so that $X(T)=\begin{bmatrix} 4 & 2 & 1\\ 1 & 1 & 1\end{bmatrix}$. If we add the row $\begin{bmatrix} r_1 & r_2 & r_3 \end{bmatrix} = \begin{bmatrix}  1& 2 &2\end{bmatrix}$ above $T$ to get $T^\prime = \begin{bmatrix} 1& 2 &2\\ 2 & 1 & 1\\ 1 & 1 & 1\end{bmatrix}$, then one may count path weights in $G_{3,3}^\Gamma(T^\prime)$ to obtain $$X(T^\prime) = \begin{bmatrix} 15 & 6 & 2\\ 4 & 2 & 1 \\ 1 & 1 & 1\end{bmatrix}.$$

Notice that \begin{align*} x_{01} &= \frac{\det X[\{1,2\},\{1,2\}]}{\det X[\{1,2\},\{1,2\}]}r_1 + \frac{\det X[\{1,2\},\{1,3\}]}{\det X[\{1,2\},\{2,3\}]}r_2 + \frac{\det X[\{1\},\{1\}]}{\det X[\{1\},\{3\}]}r_3 \\&= 1 + \left(\frac{3}{1}\right) 2+ \left(\frac{4}{1}\right)2 \\&=15,\end{align*} which agrees with the $j=1$ case of Equation~(\ref{BorderingTheoremEquation}).
	
\end{ex}

Similar results can be obtained for the addition of a row or column to the other sides of a TP matrix. For later use, we record this result for the addition of a row below $X$.

\begin{thm} \label{BorderingTheorem2}

Let $X$ be an $m\times n$ TP matrix with \emph{\reflectbox{L}}-scaffolding $T$. Suppose $X^\prime$ is an $(m+1)\times n$ TP matrix obtained from $X$ by adding $\begin{bmatrix}x_{m+1,1} & x_{m+1,2} & \cdots & x_{m+1,n}\end{bmatrix}$ beneath the $m$th row. Then there exists a strongly positive $\vect{q}=\begin{bmatrix} q_1 & \cdots & q_n\end{bmatrix}$ such that $T^\prime = \begin{bmatrix} T \\ \vect{q}\end{bmatrix}$ is the \emph{\reflectbox{L}}-scaffolding of $X^\prime$ and for all $j\in [n]$, \begin{align}x_{m+1,j} &= \sum_{i=1}^j \sum_{\substack{P:(m+1,i)\to j, \\P\in G_{m+1,n}^{\reflectbox{\emph{\SMALL L}}}(T^\prime)}} w(P)q_i \nonumber \\ &=\sum_{i=1}^j \frac{\det X[\{\cdots m-1, m\},\{\cdots i-1, j\}]}{\det X[\{\cdots m-1, m \},\{ \cdots i-1, i\}]}q_i.\label{BorderingTheoremEquation2} \end{align}
Conversely, if we take positive real numbers $q_1,q_2,\ldots,q_n$ and define $x_{m+1,j}$ as in Equation~(\ref{BorderingTheoremEquation2}), then the matrix $X^\prime$ obtained from $X$ by adding $$\begin{bmatrix}x_{m+1,1} & x_{m+1,2} & \cdots & x_{m+1,n}\end{bmatrix}$$ below the $m$th row is totally positive. 
\end{thm}

\subsection{Row Insertion}

Let $X$ be an $m\times n$ TP matrix and suppose we wish to insert a new row between rows $k$ and $k+1$ of $X$ while maintaining total positivity. We will index this new row by $k^\prime$. We use the bordering results above to show that the possible inserted rows correspond to strongly positive solutions to a homogeneous system of $2n$ linear equations in $3n$ unknowns. The result will be Theorem~\ref{InsertionTheorem} below but let us derive these equations before stating the theorem. 

We begin by finding the first $n$ equations. Let $X_1$ be the submatrix of $X$ consisting of the first $k$ rows of $X$ and let $X_2$ be the submatrix of $X$ consisting of rows $k+1$ through $m$ of $X$. Obviously $X_1$ and $X_2$ are themselves totally positive. 

The insertion of a new row between rows $k$ and $k+1$ of $X$ is simultaneously the addition of a row beneath $X_1$ and the addition of a row above $X_2$.  If $\vect{r}=\begin{bmatrix} r_1 & \cdots & r_n\end{bmatrix}$ is a strongly positive row we add above the $\Gamma$-scaffolding of $X_2$ and $\vect{q }= \begin{bmatrix} q_1 & \cdots & q_n\end{bmatrix}$ is a strongly positive row we add below the $\reflectbox{L}$-scaffolding of $X_1$, then by Theorems~\ref{BorderingTheorem1} and~\ref{BorderingTheorem2}, it is necessary that for all $j\in [n]$, \begin{multline*}\sum_{\ell=j}^n \frac{\det X_2[\{k+1\,\,k+2\cdots\},\{j\, \ell+1\cdots\}]}{\det X_2[\{k+1\,\, k+2\cdots\},\{\ell\,\, \ell+1\cdots\}]}r_\ell \\=  \sum_{i=1}^j \frac{\det X_1[\{ \cdots k-1\,\, k\},\{\cdots i-1\,\, j\}]}{\det X_1[\{\cdots k-1\,\,k\},\{\cdots i-1\,\,i\}]}q_i.\end{multline*}

Since all minors in the above equation are equal to the corresponding minor in $X$, we may write these equations as \begin{multline} \sum_{\ell=j}^n \frac{\det X[\{k+1\,\,k+2\cdots\},\{j\, \ell+1\cdots\}]}{\det X[\{k+1\,\, k+2\cdots\},\{\ell\,\, \ell+1\cdots\}]}r_\ell \\=  \sum_{i=1}^j \frac{\det X[\{ \cdots k-1\,\, k\},\{\cdots i-1\,\, j\}]}{\det X[\{\cdots k-1\,\,k\},\{\cdots i-1\,\,i\}]}q_i.
 \label{EquationsSet1}	
 \end{multline}

\begin{ex} \label{InsertionExample}

We will keep a running example and show how to insert a row after the second row in $$X=\begin{bmatrix} 6 & 3 & 1\\ 3 & 2 & 1\\1 & 1 & 1\end{bmatrix}.$$

Instead of writing down the Equations~(\ref{EquationsSet1}), we will follow their derivation described above. First we find the \reflectbox{L}-scaffolding $T_1$ of $X_1= \begin{bmatrix} 6 & 3 & 1\\ 3 & 2 & 1\end{bmatrix}$ and the $\Gamma$-scaffolding $T_2$ of $X_2=\begin{bmatrix} 1 & 1 & 1\end{bmatrix}$.

Using Cauchon's Algorithms, we obtain $$T_1= \begin{bmatrix} 6 & 3 & 1\\3 & \frac{1}{2}& \frac{1}{3}\end{bmatrix} \hspace{0.5cm}  \textrm{and}\hspace{0.5cm} T_2=\begin{bmatrix} 1 & 1 & 1\end{bmatrix}.$$ In the notation of Theorems~\ref{BorderingTheorem1} and~\ref{BorderingTheorem2}, we now consider  $G_{1,3}^\Gamma(T_2^\prime)$ and $G_{2,3}^{\SMALL\reflectbox{L}}(T_1^\prime)$ which are drawn Figure~\ref{InsertionExampleFigure1}. 

 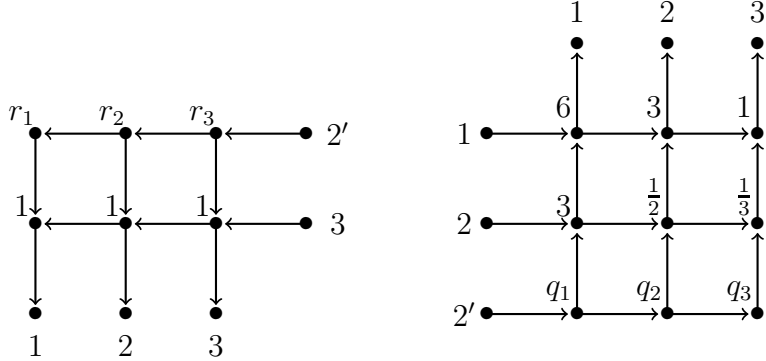
\begin{figure}[ht]
\begin{tikzpicture}[xscale=1.2, yscale=1.2]

\node at (0,0) {$\bullet$};
\node at (0,-0.35) {$1$};

\node at (1,0) {$\bullet$};
\node at (1,-0.35) {$2$};

\node at (2,0) {$\bullet$};
\node at (2,-0.35) {$3$};

\node at (3,2) {$\bullet$};
\node at (3.35,2) {$2^\prime$};

\node at (3,1) {$\bullet$};
\node at (3.35,1) {$3$};

\node at (0,1) {$\bullet$};
\node at (-0.15,1.2) {$1$};
\node at (1,1) {$\bullet$};
\node at (0.85,1.2) {$1$};
\node at (2,1) {$\bullet$};
\node at (1.85,1.2) {$1$};

\node at (0,2) {$\bullet$};
\node at (-0.15,2.2) {$r_1$};
\node at (1,2) {$\bullet$};
\node at (0.85,2.2) {$r_2$};
\node at (2,2) {$\bullet$};
\node at (1.85,2.2) {$r_3$};

\draw [->, thick, black] (3,2) -- (2.1,2);
\draw [->, thick, black] (2,2) -- (1.1,2);
\draw [->, thick, black] (1,2) -- (0.1,2);


\draw [->, thick, black] (3,1) -- (2.1,1);
\draw [->, thick, black] (2,1) -- (1.1,1);
\draw [->, thick, black] (1,1) -- (0.1,1);

\draw [->, thick, black] (0,2) -- (0,1.1);
\draw [->, thick, black] (0,1) -- (0,0.1);

\draw [->, thick, black] (1,2) -- (1,1.1);
\draw [->, thick, black] (1,1) -- (1,0.1);

\draw [->, thick, black] (2,2) -- (2,1.1);
\draw [->, thick, black] (2,1) -- (2,0.1);

%
%
%
%
%
%
%
%
%
%
%
%
%
%
%
%
%
%
%
%


\node at (6,3) {$\bullet$};
\node at (6,3.35) {$1$};

\node at (7,3) {$\bullet$};
\node at (7,3.35) {$2$};

\node at (8,3) {$\bullet$};
\node at (8,3.35) {$3$};

\node at (5,2) {$\bullet$};
\node at (4.75,2) {$1$};

\node at (5,1) {$\bullet$};
\node at (4.75,1) {$2$};

\node at (5,0) {$\bullet$};
\node at (4.75,0) {$2^\prime$};

\node at (6,1) {$\bullet$};
\node at (5.85,1.2) {$3$};
\node at (7,1) {$\bullet$};
\node at (6.85,1.3) {$\frac{1}{2}$};
\node at (8,1) {$\bullet$};
\node at (7.85,1.3) {$\frac{1}{3}$};

\node at (6,2) {$\bullet$};
\node at (5.85,2.3) {$6$};
\node at (7,2) {$\bullet$};
\node at (6.85,2.3) {$3$};
\node at (8,2) {$\bullet$};
\node at (7.85,2.3) {$1$};

\node at (6,0) {$\bullet$};
\node at (5.8,0.25) {$q_1$};
\node at (7,0) {$\bullet$};
\node at (6.8,0.25) {$q_2$};
\node at (8,0) {$\bullet$};
\node at (7.8,0.25) {$q_3$};

\draw [->, thick, black] (5,2) -- (5.9,2);
\draw [->, thick, black] (6,2) -- (6.9,2);
\draw [->, thick, black] (7,2) -- (7.9,2);

\draw [->, thick, black] (5,1) -- (5.9,1);
\draw [->, thick, black] (6,1) -- (6.9,1);
\draw [->, thick, black] (7,1) -- (7.9,1);

\draw [->, thick, black] (5,0) -- (5.9,0);
\draw [->, thick, black] (6,0) -- (6.9,0);
\draw [->, thick, black] (7,0) -- (7.9,0);

\draw  [->, thick, black] (6,0) -- (6,0.9);
\draw  [->, thick, black] (6,1) -- (6,1.9);
\draw  [->, thick, black] (6,2) -- (6,2.9);

\draw  [->, thick, black] (7,0) -- (7,0.9);
\draw  [->, thick, black] (7,1) -- (7,1.9);
\draw  [->, thick, black] (7,2) -- (7,2.9);

\draw  [->, thick, black] (8,0) -- (8,0.9);
\draw  [->, thick, black] (8,1) -- (8,1.9);
\draw  [->, thick, black] (8,2) -- (8,2.9);

%
%
%
%
%
%

\end{tikzpicture}

\caption{Graphs $G_{1,3}^\Gamma(T_2^\prime)$ (left) and $G_{2,3}^{\SMALL \reflectbox{L}}(T_1\prime)$ (right) with new row $2^\prime$ added above and below. }\label{InsertionExampleFigure1} 
	\end{figure}

Our first set of equations then comes from requiring that the sum over paths from row vertex $2^\prime$ to column vertex $j$ are the same for each $j$. This yields the equations \begin{align} r_1+r_2+r_2 &= q_1 \label{InsertionExampleEqn1} \\
 r_2+r_3 &= \frac{2}{3}q_1+q_2 \\ r_3&=\frac{1}{3}q_1+q_2+q_3 \label{InsertionExampleEqn3}.	
 \end{align}

\end{ex}

Now we find the next $n$ equations. The problem is that while a strongly positive solution to Equations~(\ref{EquationsSet1}) is necessary, it is not sufficient. Indeed, if we do take a strongly positive solution and form $X^\prime$ by inserting the row  $k^\prime$ so determined, we are not guaranteed that minors of $X^\prime$ involving row $k^\prime$ and rows in both $X_1$ and $X_2$ are positive.

The question then is which strongly positive solutions of the Equations~(\ref{EquationsSet1}) guarantee total positivity of $X^\prime$? 

Consider the application of Cauchon's Algorithm to $X^\prime$ up through step $(k+1,1)$. The resulting matrix $(X^\prime)^{(k+1,1)}$ has the block form $$(X^\prime)^{(k+1,1)} = \begin{bmatrix} \hat{X} \\ \mathbf{r}\\ T_2\end{bmatrix}$$ where note that $\hat{X}$ also equals the first $k$ rows of $X^{(k+1,1)}$. Equivalently, $\hat{X}$ is the TP matrix whose $\Gamma$-scaffolding consists of the first $k$ rows of the $\Gamma$-scaffolding of $X$.

Since $X_2$ is TP, we already know that $T_2$ is positive. By Theorem~\ref{CauchonTheorem} it follows that $X^\prime$ is TP if and only if $$\begin{bmatrix} \hat{X} \\ \mathbf{r}\end{bmatrix}$$ is TP. 

Now, since $X$ is TP so is $\hat{X}$, again by Theorem~\ref{CauchonTheorem}. Therefore by Theorem~\ref{BorderingTheorem2}, $$\begin{bmatrix} \hat{X} \\ \mathbf{r}\end{bmatrix}$$ is TP if and only if there is a strongly positive $\vect{s} = \begin{bmatrix} s_1 & \cdots & s_n\end{bmatrix}$ with \begin{align} r_j &= \sum_{i=1}^j \frac{\det \hat{X}[\{\cdots k-1, k\},\{\cdots i-1, j\}]}{\det \hat{X}[\{\cdots k-1, k \},\{ \cdots i-1, i\}]}s_i.\label{EquationsSet2}\end{align}

\begin{ex}

We continue Example~\ref{InsertionExample}. For the second set of equations we need $\hat{X}$. This may come from applying the first three steps of Cauchon's Algorithm to $X$ and extracting the first two rows, or taking the $\Gamma$-scaffolding $T$ of $X$ (which is the $3\times 3$ matrix with every entry $1$) and setting $\hat{X}$ to be the matrix whose $\Gamma$-scaffolding are the first two rows of $T$. Either way, one will find $$\hat{X}=\begin{bmatrix} 3 & 2 & 1\\ 1 & 1 & 1 \end{bmatrix}.$$ 

The \reflectbox{L}-scaffolding $\hat{T}$ of $\hat{X}$ is $$\hat{T}=\begin{bmatrix} 3 & 2 & 1\\ 1 & \frac{1}{3}&\frac{1}{2}\end{bmatrix}.$$ The graph $G_{2,3}^{\reflectbox{\SMALL L}}(\hat{T}^\prime)$ is illustrated in Figure~\ref{InsertionExampleFigure2}.
 \begin{figure}[ht]
\begin{tikzpicture}[xscale=1.2, yscale=1.2]

\node at (6,3) {$\bullet$};
\node at (6,3.35) {$1$};

\node at (7,3) {$\bullet$};
\node at (7,3.35) {$2$};

\node at (8,3) {$\bullet$};
\node at (8,3.35) {$3$};

\node at (5,2) {$\bullet$};
\node at (4.75,2) {$1$};

\node at (5,1) {$\bullet$};
\node at (4.75,1) {$2$};

\node at (5,0) {$\bullet$};
\node at (4.75,0) {$2^\prime$};

\node at (6,1) {$\bullet$};
\node at (5.85,1.2) {$1$};
\node at (7,1) {$\bullet$};
\node at (6.85,1.3) {$\frac{1}{3}$};
\node at (8,1) {$\bullet$};
\node at (7.85,1.3) {$\frac{1}{2}$};

\node at (6,2) {$\bullet$};
\node at (5.85,2.3) {$3$};
\node at (7,2) {$\bullet$};
\node at (6.85,2.3) {$2$};
\node at (8,2) {$\bullet$};
\node at (7.85,2.3) {$1$};

\node at (6,0) {$\bullet$};
\node at (5.8,0.2) {$s_1$};
\node at (7,0) {$\bullet$};
\node at (6.8,0.2) {$s_2$};
\node at (8,0) {$\bullet$};
\node at (7.8,0.2) {$s_3$};

\draw [->, thick, black] (5,2) -- (5.9,2);
\draw [->, thick, black] (6,2) -- (6.9,2);
\draw [->, thick, black] (7,2) -- (7.9,2);

\draw [->, thick, black] (5,1) -- (5.9,1);
\draw [->, thick, black] (6,1) -- (6.9,1);
\draw [->, thick, black] (7,1) -- (7.9,1);

\draw [->, thick, black] (5,0) -- (5.9,0);
\draw [->, thick, black] (6,0) -- (6.9,0);
\draw [->, thick, black] (7,0) -- (7.9,0);

\draw  [->, thick, black] (6,0) -- (6,0.9);
\draw  [->, thick, black] (6,1) -- (6,1.9);
\draw  [->, thick, black] (6,2) -- (6,2.9);

\draw  [->, thick, black] (7,0) -- (7,0.9);
\draw  [->, thick, black] (7,1) -- (7,1.9);
\draw  [->, thick, black] (7,2) -- (7,2.9);

\draw  [->, thick, black] (8,0) -- (8,0.9);
\draw  [->, thick, black] (8,1) -- (8,1.9);
\draw  [->, thick, black] (8,2) -- (8,2.9);

\end{tikzpicture}

\caption{The \reflectbox{L}-scaffolding of $\hat{X}$ with new row added below.}\label{InsertionExampleFigure2} 
	\end{figure}
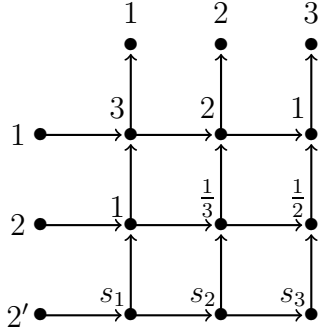
	
Therefore, we seek a strongly positive solution to Equations~(\ref{InsertionExampleEqn1})-(\ref{InsertionExampleEqn3}) such that there are positive  $s_1,s_2,s_3$ with \begin{align*} r_1 &= s_1\\
 r_2 &= s_1 + s_2\\ r_3 &= s_1+2s_2+s_3.
 \end{align*}

For example, one strongly positive solution is \begin{eqnarray*}
r_1=1, & r_2=2, & r_3=6,\\
q_1=9, & q_2=2, & q_3=1,\\
s_1=1, & s_2=1, & s_3=3.	
\end{eqnarray*}
This solution gives the TP matrix $$X^\prime = \begin{bmatrix} 6 & 3 & 1\\ 3 & 2 & 1\\ 9 & 8 & 6\\ 1 & 1 & 1\end{bmatrix}.$$

\end{ex}

Our discussion has led us to the following result.

\begin{thm} \label{InsertionTheorem}

Let $X$ be an $m\times n$ TP matrix. Let $\hat{X}$ be the TP matrix whose $\Gamma$-scaffolding consists of the first $k$ rows of the $\Gamma$-scaffolding of $X$. 

Any insertion of a new row between rows $k$ and $k+1$ of $X$ that maintains total positivity corresponds to a strongly positive solution $$\begin{bmatrix} r_1 & \cdots & r_n & q_1 & \cdots & q_n & s_1 & \cdots & s_n\end{bmatrix}$$ to the following  system of linear equations:
\begin{multline}\label{EquationsSet3} \sum_{\ell=j}^n \frac{\det X[\{k+1\,\,k+2\cdots\},\{j\, \ell+1\cdots\}]}{\det X[\{k+1\,\, k+2\cdots\},\{\ell\,\, \ell+1\cdots\}]}r_\ell \\=  \sum_{i=1}^j \frac{\det X[\{ \cdots k-1\,\, k\},\{\cdots i-1\,\, j\}]}{\det X[\{\cdots k-1\,\,k\},\{\cdots i-1\,\,i\}]}q_i
 \end{multline}
\begin{align} r_j &= \sum_{i=1}^j \frac{\det \hat{X}[\{\cdots k-1, k\},\{\cdots i-1, j\}]}{\det \hat{X}[\{\cdots k-1, k \},\{ \cdots i-1, i\}]}s_i,\label{EquationsSet4}\end{align}
where $j$ runs from $1$ to $n$ in both~(\ref{EquationsSet3}) and~(\ref{EquationsSet4}). 

If $\begin{bmatrix} x_{k^\prime 1} & \cdots & x_{k^\prime n}\end{bmatrix}$ is the row inserted, then $x_{k^\prime j}$ equals the common value in the $j$th equation of~(\ref{EquationsSet3}).
\end{thm}

That there exist strongly positive solutions to the system of equations in Theorem~\ref{InsertionTheorem} may already be inferred by the main result of~\cite{johnsonsmith}. However, in the spirit of independence and in keeping with the theme of the present work, we provide an alternative approach.

\begin{thm} \label{ExistenceTheorem}
There always exist strongly positive solutions to the system of equations in the statement of Theorem~\ref{InsertionTheorem}.

\end{thm}

\begin{proof}

For any choice of positive $s_1,\ldots, s_n$, each $r_j$ is positive. So we only need to show that there always exists an appropriate choice of positive $s_1,\ldots, s_n$ so that each $q_1,\ldots,q_n$ is positive. 

To do this, we show that for each $j$, the expression for $q_j$ as a linear combination of $s_1,\ldots, s_n$ is such that the coefficient of $s_n$ is positive. So for any choice of positive $s_1,\ldots, s_{n-1}$, we then just choose $s_n$ large enough to make $q_j>0$. 

To find the coefficient of $s_n$ in the expression for $q_j$, we set $s_1=\cdots= s_{n-1}=0$ and, for convenience,  $s_n=x_{k+1,n}>0$. From Equations~(\ref{EquationsSet4}) we have $r_1=\cdots = r_{n-1}=0$ and $r_n=s_n=x_{k+1,n}$. 

On the other hand, let $\overline{X} = \begin{bmatrix} X_1 \\ \vect{x}_{k+1}\end{bmatrix}$, where, as above, $X_1$ consists of the first $k$ rows of $X$ and $x_{k+1}$ is the $(k+1)$st row of $X$. Let $\overline{T}=[\overline{t}_{ij}]$ be the \reflectbox{L}-scaffolding of $\overline{X}$. We prove by induction on $j$ that with the above choice of $s_1,\ldots, s_n$, one obtains $$q_j=\overline{t}_{k+1,j}.$$ In other words, in the expression for $q_j$ as a linear combination of $s_1,\ldots,s_n$, the coefficient of $s_n$ is $\frac{\overline{t}_{k+1,j}}{x_{k+1,n}}>0,$ which will complete the proof.

To implement our strategy, we use the first form of the equation in Theorem~\ref{BorderingTheorem2} to replace the right sides in Equations~(\ref{EquationsSet3}). We obtain $$\frac{x_{k+1,j}}{x_{k+1,n}}r_n = \sum_{i=1}^j \sum_{\substack{P:(k+1, i)\to j \\ P\in G_{k+1,n}^{\reflectbox{\SMALL L}}(\overline{T})}} w(P)q_i,$$ and since $r_n=x_{k+1,n}$, $$x_{k+1,j} = \sum_{i=1}^j \sum_{\substack{P:(k+1, i)\to j \\ P\in G_{k+1,n}^{\reflectbox{\SMALL L}}(\overline{T})}} w(P)q_i.$$

For $j=1$, this reduces to simply $x_{k+1,1}=q_1$. Since in the $\reflectbox{L}$-scaffolding of $\overline{X}$, one has $x_{k+1,1}=\overline{t}_{k+1,1}$, our assertion holds in this case.

Now suppose $q_i= \overline{t}_{k+1,i}$ for $1\leq i\leq j-1$. Then if $P:(k+1,i)\to j$ is a path in  $G_{k+1,n}^{\reflectbox{\SMALL L}}(\overline{T})$ for some $1\leq i\leq j-1$, then $w(P)q_i = \overline{t}_{k+1,i}w(P)$ is the weight of a path $Q:k+1\to j$, and conversely every such path arises in this way. Thus $$\sum_{i=1}^{j-1} \sum_{\substack{P:(k+1, i)\to j \\ P\in G_{k+1,n}^{\reflectbox{\SMALL L}}(\overline{T})}} w(P)q_i = \sum_{i=1}^j \sum_{\substack{P:(k+1, i)\to j \\ P\in G_{k+1,n}^{\reflectbox{\SMALL L}}(\overline{T})}} \overline{t}_{k+1,i}w(P)$$ is precisely the sum of the weights of all paths from $k+1$ to $i$ except the primary path from $k+1$ to $i$. Since $x_{k+1,i}$ is the sum of the weights of \emph{all} paths from $k+1$ to $i$, it follows that we must have $q_j=\overline{t}_{k+1,j}$. As explained above, this suffices to complete the proof.

\end{proof}
	
Practically, finding a line insertion may now be done using standard linear programming techniques or even Dines' old algorithm~\cite{dines}. One may also easily extract an algorithm from the above proof: First Choose any positive $s_1,\ldots, s_{n-1}$, then write each $q_j$ in terms of $s_1,\ldots, s_n$ using the above equations, and finally choose $s_n$ large enough to make each $q_j>0$.

We hope the concept of scaffoldings will prove a fruitful method for other TP completion problems. For example, many of the main results of~\cite{fallatjohnson} may also be recovered using this approach, and in a sequel paper we will address the echelon completion problem described in~\cite{johnsonwei}.

\bibliography{casteels1}

\providecommand{\bysame}{\leavevmode\hbox to3em{\hrulefill}\thinspace}
\providecommand{\MR}{\relax\ifhmode\unskip\space\fi MR }
\providecommand{\MRhref}[2]{%
  \href{http://www.ams.org/mathscinet-getitem?mr=#1}{#2}
}
\providecommand{\href}[2]{#2}
\begin{thebibliography}{10}

\bibitem{adm}
Mohammad Adm, Khawla Al~Muhtaseb, Ayed~Abedel Ghani, Shaun Fallat, and J\"{u}rgen Garloff, \emph{Further applications of the {C}auchon algorithm to rank determination and bidiagonal factorization}, Linear Algebra Appl. \textbf{545} (2018), 240--255. \MR{3769121}

\bibitem{adm2}
Mohammad Adm and J\"{u}rgen Garloff, \emph{Improved tests and characterizations of totally nonnegative matrices}, Electron. J. Linear Algebra \textbf{27} (2014), 588--610. \MR{3266168}

\bibitem{casteels}
Karel Casteels, \emph{Quantum matrices by paths}, Algebra Number Theory \textbf{8} (2014), no.~8, 1857--1912. \MR{3285618}

\bibitem{cauchon}
G\'{e}rard Cauchon, \emph{Effacement des d\'{e}rivations et spectres premiers des alg\`ebres quantiques}, J. Algebra \textbf{260} (2003), no.~2, 476--518. \MR{1967309}

\bibitem{dines}
Lloyd~L. Dines, \emph{On positive solutions of a system of linear equations}, Ann. of Math. (2) \textbf{28} (1926/27), no.~1-4, 386--392. \MR{1502792}

\bibitem{tnnbook}
Shaun~M. Fallat and Charles~R. Johnson, \emph{Totally nonnegative matrices}, Princeton Series in Applied Mathematics, Princeton University Press, Princeton, NJ, 2011. \MR{2791531}

\bibitem{fallatjohnson}
Shaun~M. Fallat, Charles~R. Johnson, and Ronald~L. Smith, \emph{The general totally positive matrix completion problem with few unspecified entries}, Electron. J. Linear Algebra \textbf{7} (2000), 1--20. \MR{1737244}

\bibitem{gesselviennot}
Ira Gessel and G\'{e}rard Viennot, \emph{Binomial determinants, paths, and hook length formulae}, Adv. in Math. \textbf{58} (1985), no.~3, 300--321. \MR{815360}

\bibitem{johnsonsmith}
Charles~R. Johnson and Ronald~L. Smith, \emph{Line insertions in totally positive matrices}, J. Approx. Theory \textbf{105} (2000), no.~2, 305--312. \MR{1775151}

\bibitem{johnsonwei}
Charles~R. Johnson and Zhen Wei, \emph{Asymmetric {TP} and {TN} completion problems}, Linear Algebra Appl. \textbf{438} (2013), no.~5, 2127--2135. \MR{3005280}

\bibitem{launois}
S.~Launois and T.~H. Lenagan, \emph{From totally nonnegative matrices to quantum matrices and back, via {P}oisson geometry}, Perspectives in {L}ie theory, Springer INdAM Ser., vol.~19, Springer, Cham, 2017, pp.~443--461. \MR{3751138}

\bibitem{pinkus}
Allan Pinkus, \emph{Totally positive matrices}, Cambridge Tracts in Mathematics, vol. 181, Cambridge University Press, Cambridge, 2010. \MR{2584277}

\bibitem{postnikov}
A.~Postnikov, \emph{Total positivity, grassmannians, and networks}, 2006, http://arxiv.org/abs/0609764.

\end{thebibliography}
\bibliographystyle{amsplain}

\end{document}